\documentclass[12pt]{amsart}

\usepackage{amsthm, amsmath, amssymb, verbatim}

\numberwithin{equation}{section}
	     
\newtheorem{thm}{Theorem}[section] 
\newtheorem{prop}{Proposition}[section]

\newtheorem{cor}{Corollary}[section]
     
\theoremstyle{definition} 

\theoremstyle{remark} 
\newtheorem{rem}{Remark\rm}[section]

\newcommand{\ep}{\varepsilon}

\newcommand{\as}{\arcsin}
\newcommand{\asn}{\operatorname{arcsn}}
\newcommand{\sn}{\operatorname{sn}}
\newcommand{\cn}{\operatorname{cn}}
\newcommand{\dn}{\operatorname{dn}}
\newcommand{\am}{\operatorname{am}}

\begin{document}
\title[Generalized elliptic functions and their application]
{Generalized elliptic functions and their application to
a nonlinear eigenvalue problem with $p$-Laplacian}
\author{Shingo Takeuchi}
\date{}
\address{Department of General Education, Kogakuin University,
2665-1 Nakano, Hachioji, Tokyo 192-0015, JAPAN}
\email{shingo@cc.kogakuin.ac.jp}
\dedicatory{Dedicated to Professor Yoshio Yamada on occasion
of his $60^{th}$ birthday}
\thanks{This work was supported by KAKENHI (No. 20740094).}
\subjclass{Primary 34L40,\ Secondary 33E05}
	
\begin{abstract}
The Jacobian elliptic functions are generalized and applied 
to a nonlinear eigenvalue problem with $p$-Laplacian.
 The eigenvalue and the corresponding eigenfunction are 
represented in terms of common parameters,
and a complete description of the 
spectra and a closed form representation of the 
corresponding eigenfunctions are obtained. As a by-product of 
the representation, it turns out that
a kind of solution is also a solution of 
another eigenvalue problem with $p/2$-Laplacian.  
\end{abstract}

\maketitle 
     
\section{Introduction} \setcounter{equation}{0}

In this paper we generalize the Jacobian elliptic functions and
apply them to a nonlinear eigenvalue problem
\begin{equation}
\label{eq:lep}
\begin{cases}
(\phi_p(u'))'+\lambda \phi_q(u)(1-|u|^q)=0, & t \in (0,T),\\
u(0)=u(T)=0,
\end{cases} \tag{$\mathrm{PE}_{pq}$}
\end{equation}
where $T,\ \lambda>0,\ p,\ q>1$ and 
$\phi_m(s)=|s|^{m-2}s\ (s \neq 0),\ =0\ (s=0)$.

Problem \eqref{eq:lep} 
appears frequently in various articles as stationary problems.
In particular, 
the equation for $p=q=2$ is called, e.g.,
the Allen-Cahn equation, 
the Chafee-Infante equation \cite{CI},
and a bistable reaction-diffusion equation
with logistic effect. 
The equation for $p=2<q$ is said to be a 
bistable reaction-diffusion equation with 
Allee effect. 
In case $p=n$ and $q=2$ with an $n$-dimensional domain, 
an equation of this type 
is known as the Euler-Lagrange equation of 
functional related to models 
introduced by Ginzburg and Landau 
for the study of phase transitions (cf. Problem 17 in \cite{BBH}).

As to \eqref{eq:lep} for general $p>1$, we have to mention the work 
\cite{GV} by Guedda and V\'{e}ron \cite{GV}. They showed that 
if $p=q>1$ then there exists a positive increasing 
sequence $\{\lambda_n\}$ such that 
a pair of solutions $\pm u_n$ of \eqref{eq:lep} with $(n-1)$-zeros 
$z_j=jT/n\ (j=1,2,\ldots,n-1)$
bifurcates from the trivial solution at $\lambda=\lambda_n$ and 
$|u_n| \to 1$  uniformly on any compact set of $(0,T) \setminus 
\{z_1,z_2,\ldots,z_{n-1}\}$ 
as $\lambda \to \infty$. Moreover, they proved  
that if $p=q>2$ then for each $n \in \mathbb{N}$ 
there exists $\Lambda_n>\lambda_n$ such that
$\lambda>\Lambda_n$ implies $|u_n|=1$ on 
\textit{flat cores} 
$[z_{j-1}+\frac{T}{2n}(\frac{\Lambda_n}{\lambda})^{1/p},
z_{j}-\frac{T}{2n}(\frac{\Lambda_n}{\lambda})^{1/p}]\ (j=1,2,\ldots,n)$
of $u_n$, where $z_0=0$ and $z_n=T$.
This is a great contrast to case $1<p=q \le 2$,
where $|u_n|<1$ in $[0,T]$.
Since the equation in \eqref{eq:lep} is autonomous, if $u_n\ (n \ge 2)$
has flat cores, then there exists uncountable
solution with $(n-1)$-zeros near $u_n$, which is produced by 
expanding and contracting
the flat cores with preserving its total length 
$T(1-(\frac{\Lambda_n}{\lambda})^{1/p})$. In this sense, 
the $n$-th branch $(\lambda,u_n)$ bifurcating from  
$(\lambda_n,0)$ causes 
the second bifurcation at $(\Lambda_n,u_{\Lambda_n})$ for each $n \ge 2$.

The phenomena of flat core in \cite{GV} above was generalized to case $p>2$ and $q>1$ 
by the author and Yamada \cite{TY}. They also studied change in bifurcation
depending on the relation between $p$ and $q$ (as far as the first bifurcation
is concerned, their proof can be applied to case $1<p \le 2$), 
and showed that 
for each $n \in \mathbb{N}$, 
if $p>q$ then there exists a pair of solutions $\pm u_n$ 
of \eqref{eq:lep} with $(n-1)$-zeros for $\lambda>0$; 
if $p=q$ then there exists $\lambda_n>0$ 
such that \eqref{eq:lep}
has no solution with $(n-1)$-zeros for $\lambda \le \lambda_n$
and \eqref{eq:lep} has 
a pair of solutions $\pm u_n$ for $\lambda>\lambda_n$
(the same result as \cite{GV});
if $p<q$ then there exists $\lambda_n^*>0$ such that 
\eqref{eq:lep} has
no solution with $(n-1)$-zeros for $\lambda < \lambda_n^*$
and \eqref{eq:lep} has 
a pair of solutions $\pm u_n$ for $\lambda=\lambda_n^*$ and 
\eqref{eq:lep} has two pairs of solutions $\pm u_n,\ \pm v_n$
satisfying $|u_n(t)|>|v_n(t)|$ with $t \neq z_j\ (j=0,1,\ldots,n)$
for $\lambda>\lambda_n^*$.
In this sense, the point $(\lambda_n^*,u_{\lambda_n^*})$ causes
the \textit{spontaneous bifurcation}. 
In any case, each solution $u_n$
has flat cores for sufficiently large $\lambda$.

The purpose of this paper is to obtain a complete description of the 
spectra and a closed form representation of the corresponding
eigenfunctions of \eqref{eq:lep}, while
the studies \cite{GV} and \cite{TY}
above are done in the way of phase-plane analysis
and no exact solution is given there.

For the description and representation, we first recall that 
the Jacobian elliptic function $\sn(t,k)$ with 
modulus $k \in [0,1)$ satisfies 
\begin{align}
\label{eq:pe22} 
u''+u(1+k^2-2k^2u^2)=0.
\end{align}
(e.g., Example 4 of p.\,516 in the book \cite{WW} of 
Whittaker and Watson).   
Eq.\,\eqref{eq:pe22} reminds that the solution 
of nonlinear eigenvalue problem \eqref{eq:lep} with
$p=q=2$ can be represented explicitly by using $\sn(t,k)$.
Indeed, for any given $k \in (0,1)$, 
the set of eigenvalues of 
$(\mathrm{PE}_{22})$ is given by
\begin{align}
\label{eq:lambda}
\lambda_n(k)=
(1+k^2) 
\left(\frac{2nK(k)}{T}\right)^2
\end{align}
for each $n \in \mathbb{N}$, with corresponding eigenfunctions
$\pm u_{n,k}$, where
\begin{align}
\label{eq:u}
u_{n,k}(t)=
\sqrt{\frac{2k^2}{1+k^2}}
\sn{\left(\frac{2nK(k)}{T}t,k\right)}
\end{align}
and $K(k)$ is the complete elliptic integral of the first kind
$$K(k)=\int_0^1 \frac{ds}{\sqrt{(1-s^2)(1-k^2s^2)}}$$
(cf. Section 2 in \cite{BF}).
Conversely, 
all nontrivial solutions are given by Eqs.\,\eqref{eq:lambda}
and \eqref{eq:u}, and in particular, 
it follows from Eq.\,\eqref{eq:u} that 
all solutions satisfy $|u|<1$. 

In our study on \eqref{eq:lep}, after the fashion of Jacobi's 
$\sn(t,k)$, we introduce
a new transcendental function $\sn_{pq}(t,k)$ 
with modulus $k \in [0,1)$. This satisfies
\begin{gather}
\label{eq:pepq} 
(\phi_p(u'))'+\frac{q}{p^*}\phi_q(u)(1+k^q-2k^q|u|^q)=0,
\end{gather}
where $p^*:=p/(p-1)$.
Using $\sn_{pq}(t,k)$, we can obtain a complete
description of the set of eigenvalues and 
the corresponding eigenfunctions
of \eqref{eq:lep} as Eqs.\,\eqref{eq:lambda} and \eqref{eq:u} with
$$K_{pq}(k)=\int_0^1 \frac{ds}{\sqrt[p]{(1-s^q)(1-k^qs^q)}}.$$ 
It is important that 
$K_{pq}(k)$ converges to $K_{\frac{p}{2},q}(0)$ 
as $k \to 1-0$ if and only if $p>2$. Indeed,
$$\lim_{k \to 1-0}K_{pq}(k)=
\int_0^1 \frac{ds}{(1-s^q)^{\frac{2}{p}}}=K_{\frac{p}{2},q}(0).
$$
Similarly, 
$\sn_{pq}(t,k)$ converges to $\sn_{\frac{p}{2},q}(t,0)$
as $k \to 1-0$.
These convergent properties yield the existence of 
special solutions, not necessarily $|u|<1$,
and we can really construct the solutions of \eqref{eq:lep}
with flat cores. 
Moreover, $\sn_{\frac{p}{2},q}(t,0)$ 
satisfies Eq.\,\eqref{eq:pepq} with $k=0$ and $p$ 
replaced by $p/2$ as well as Eq.\,\eqref{eq:pepq} with $k=1$. 
Thus, we obtain the following (curious) property:
a kind of solution of \eqref{eq:lep}
is also a solution of the nonlinear eigenvalue problem
with $p/2$-Laplacian
\begin{equation*}
\begin{cases}
(\phi_{\frac{p}{2}}(u'))'+\lambda \phi_q(u)=0, & t \in (0,T),\\
u(0)=u(T)=0.
\end{cases}
\end{equation*}

This paper is organized as follows. In Section 2, we introduce 
a generalized trigonometric function $\sin_{pq}(t)$ given by 
Dr\'{a}bek and Man\'{a}sevich \cite{DM} and define a new 
transcendental function $\sn_{pq}(t,k)$, which is a generalization of
the Jacobian elliptic function $\sn(t,k)$ and an extension of 
$\sin_{pq}(t)$ as $\sn_{pq}(t,0)=\sin_{pq}(t)$. 
In Section 3, we apply them to nonlinear
eigenvalue problems, particularly to the problem
considered in \cite{GV} and \cite{TY},
and obtain complete
descriptions of the set of eigenvalues and 
the corresponding eigenfunctions.

\section{Transcendental Functions}

\subsection{Generalized trigonometric functions}

Generalized trigonometric functions were 
introduced by Dr\'{a}bek and Man\'{a}sevich
\cite {DM} (see also \cite{DKT}). 
For $\sigma \in [0,1]$, we define
(in a slightly different way from \cite{DM})
\begin{align}
\label{eq:as}
\as_{pq}(\sigma):=\int_0^{\sigma}\frac{ds}{(1-s^q)^{\frac{1}{p}}},
\end{align}
where $p>1,\ q>0$. Letting $s=z^{1/q}$, we have
\begin{align*}
\as_{pq}(\sigma)=\frac{1}{q}\int_0^{\sigma^q}
z^{\frac{1}{q}-1}(1-z)^{-\frac{1}{p}}\,dz
=\frac{1}{q}\tilde{B}\left(\frac{1}{q},\frac{1}{p^*},\sigma^q\right),
\end{align*}
where $\tilde{B}(s,t,u)$ denotes the incomplete beta function
\begin{align*}
\tilde{B}(s,t,u)=\int_0^u z^{s-1}(1-z)^{t-1}\,dz.
\end{align*}
We define the constant $\pi_{pq}$ as
\begin{align*}
\pi_{pq}:=2\as_{pq}(1)=\frac{2}{q}B\left(\frac{1}{q},\frac{1}{p^*}\right),
\end{align*}
where $B(s,t)$ denotes the beta function
\begin{align*}
B(s,t)=\tilde{B}(s,t,1)=\int_0^1 z^{s-1}(1-z)^{t-1}\,dz.
\end{align*}

We have that $\as_{pq}:[0,1] \to [0,\pi_{pq}/2]$, and is strictly increasing.
Let us denote its inverse by $\sin_{pq}$. Then, $\sin_{pq}:[0,\pi_{pq}/2]
\to [0,1]$ and is strictly increasing. 
We extend $\sin_{pq}$ to all $\mathbb{R}$ (and still denote this extension by
$\sin_{pq}$) in the following form: for $t \in [\pi_{pq}/2,\pi_{pq}]$,
we set $\sin_{pq}{(t)}:=\sin_{pq}{(\pi_{pq}-t)}$, then for $t \in [-\pi_{pq},0]$,
we define $\sin_{pq}{(t)}:=-\sin_{pq}{(-t)}$, and finally we extend 
$\sin_{pq}$ to all $\mathbb{R}$ as a $2\pi_{pq}$ periodic function. 

When $0<p \le 1$, we also define $\as_{pq}$ as Eq.\,\eqref{eq:as} 
for $\sigma \in [0,1)$. 
We have that $\as_{pq}:[0,1) \to [0,\infty)$, and is strictly increasing.
Let us denote its inverse by $\sin_{pq}$. Then, $\sin_{pq}:[0,\infty)
\to [0,1)$ and is strictly increasing. 
We extend $\sin_{pq}$ to all $\mathbb{R}$
as $\sin_{pq}{(t)}:=-\sin_{pq}{(-t)}$
for $t \in (-\infty,0]$ and still denote this extension by 
$\sin_{pq}$. 

\begin{rem}
We immediately find that
$\sin_{22}{(t)}=\sin{(t)}$ and $\pi_{22}=\pi$
from the properties of the beta function. Moreover, 
$\sin_{pp}{(t)}=\sin_p{(t)}$ and $\pi_{pp}=\pi_p=\frac{2\pi}{p\sin{\frac{\pi}{p}}}$,
where $\sin_p$ and $\pi_p$ are the generalized sine function
and its half-period, respectively,
appearing in \cite{Do}, \cite{DoR} and \cite{DKT}.
\end{rem}

We define for $t \in [0,\pi_{pq}/2]$ (in case $0<p \le 1$, for $t \in [0,\infty)$)
\begin{align*}
\cos_{pq}{(t)}:=(1-\sin_{pq}^q{(t)})^{\frac{1}{p}},
\end{align*}
then we obtain
\begin{align*}
&\cos_{pq}^p{(t)}+\sin_{pq}^q{(t)}=1,\\
&\frac{d}{dt}\sin_{pq}{(t)}=\cos_{pq}{(t)}.
\end{align*}

\begin{prop}
For $p,\ q>1$, $\sin_{pq}$ satisfies 
for all $\mathbb{R}$
\begin{equation}
\label{eq:sin}
(\phi_p(u'))'+\frac{q}{p^*}\phi_q(u)=0.
\end{equation}
\end{prop}

\begin{proof}
For $t \in (0,\pi_{pq}/2)$ we have
\begin{align*}
(\phi_p(u'))'&=(\phi_p(\cos_{pq}{(t)}))'\\
&=((1-\sin^q_{pq}{(t)})^{\frac{1}{p^*}})'\\
&=\frac{1}{p^*}(1-\sin^q_{pq}{(t)})^{-\frac{1}{p}}\cdot (-q\sin^{q-1}_{pq}{(t)})
\cdot \cos_{pq}{(t)}\\
&=-\frac{q}{p^*}\phi_q(u).
\end{align*}
By symmetry of $\sin_{pq}$, Eq.\,\eqref{eq:sin} holds true 
for $t \neq t_n:=n \pi_{pq}/2,\ n \in \mathbb{Z}$.
Since $\lim_{t \to t_n} (\phi_p(u'))'$ exists, 
$\phi_p(u')$ is differentiable also at $t=t_n$ 
and satisfies Eq.\,\eqref{eq:sin} for all $\mathbb{R}$ in the classical sense.
\end{proof}

\subsection{Generalized Jacobian elliptic functions}

We shall introduce new transcendental functions, which generalize the Jacobian 
elliptic functions. 
For $\sigma \in [0,1]$ and $k \in [0,1)$, we define
\begin{align}
\label{eq:asn}
\asn_{pq}{(\sigma)}=\asn_{pq}{(\sigma,k)}
:=\int_0^{\sigma} \frac{ds}{\sqrt[p]{(1-s^q)(1-k^qs^q)}},
\end{align}
where $p>1,\ q>0$. We define the constant $K_{pq}(k)$ as
\begin{align*}
K_{pq}=K_{pq}(k):=\asn_{pq}(1,k)=\int_0^1 \frac{ds}{\sqrt[p]{(1-s^q)(1-k^qs^q)}}
\end{align*}  

We have that $\asn_{pq}:[0,1] \to [0,K_{pq}]$, and is strictly increasing.
Let us denote its inverse by $\sn_{pq}(\cdot)=\sn_{pq}(\cdot,k)$. Then, 
$\sn_{pq}:[0,K_{pq}] \to [0,1]$ and is strictly increasing. 
We extend $\sn_{pq}$ to all $\mathbb{R}$ (and still denote this extension by
$\sn_{pq}$) in the following form: for $t \in [K_{pq},2K_{pq}]$,
we set $\sn_{pq}{(t)}:=\sn_{pq}{(2K_{pq}-t)}$, then for $t \in [-2K_{pq},0]$,
we define $\sn_{pq}{(t)}:=-\sn_{pq}{(-t)}$, and finally we extend 
$\sn_{pq}$ to all $\mathbb{R}$ as a $4K_{pq}$ periodic function. 

When $0<p \le 1$, we also define $\asn_{pq}$ as Eq.\,\eqref{eq:asn} 
for $\sigma \in [0,1)$. 
We have that $\asn_{pq}:[0,1) \to [0,\infty)$, and
is strictly increasing. Let us denotes its inverse by 
$\sn_{pq}(\cdot)=\sn_{pq}(\cdot,k)$. Then, 
$\sn_{pq}:[0,\infty) \to [0,1)$ and is strictly increasing. 
We extend $\sn_{pq}$ to all $\mathbb{R}$ as
$\sn_{pq}{(t)}:=-\sn_{pq}{(-t)}$ for $t \in (-\infty,0]$
and still denote this extension by
$\sn_{pq}$.

The following proposition is crucial to our study.
\begin{prop}
\label{prop:K}
For $p,\ q>0$, $K_{pq}$ is continuous and strictly increasing in $[0,1)$,
$2K_{pq}(0)=\pi_{pq}$ and $\sn_{pq}(t,0)=\sin_{pq}{(t)}$.
Moreover,
\begin{align*}
&\lim_{k \to 1-0}2K_{pq}(k)=
\begin{cases}
\pi_{\frac{p}{2},q} \quad & \text{if}\;\ p>2,\\
\infty & \text{if}\;\ 0 <p \le 2,
\end{cases}\\
&\lim_{k \to 1-0}\sn_{pq}(t,k)=\sin_{\frac{p}{2},q}{(t)}.
\end{align*}
\end{prop}

\begin{proof}
The first half is trivial from the definitions of $K_{pq}$ and $\sn_{pq}$.
If $p>2$, then the monotone convergence theorem of Beppo Levi gives
$$\lim_{k \to 1-0}2K_{pq}(k)=2\int_0^1 \frac{ds}{(1-s^q)^{\frac{2}{p}}}=
2\arcsin_{\frac{p}{2},q}(1)=\pi_{\frac{p}{2},q}.$$ 
If $0<p \le 2$, then $2K_{pq}(k)$ diverges to $\infty$ as $k \to 1-0$
by Fatou's lemma. 

The last property is proved as follows.
By the symmetry of $\sn_{pq}(\cdot,k)$, we may assume $t>0$.    
Suppose $p>2$ and that 
there exist $t_0,\ \ep>0$ and $\{k_j\}$ such that
$k_j \to 1$ as $j \to \infty$ and 
\begin{align}
\label{eq:conti}
|\sigma_{k_j}-\sin_{\frac{p}{2},q}(t_0)| \ge \ep,
\end{align}
where $\sigma_{k_j}=\sn_{pq}{(t_0,k_j)}$.
Let $n \in \mathbb{Z}$ be the number satisfying
$t_0 \in I_n:=[n\pi_{\frac{p}{2},q}/2,
(n+1)\pi_{\frac{p}{2},q}/2)$ and $j \in \mathbb{N}$ a large number
satisfying $t_0 \in I_n(k_j):=[nK_{pq}(k_j),(n+1)K_{pq}(k_j))$. 
We write $\sn_{pq}^{(n)}(\cdot,k_j)$ 
as $\sn_{pq}(\cdot,k_j)$ on 
$I_n(k_j)$
and $\sin_{pq}^{(n)}(\cdot)$ as 
$\sin_{pq}(\cdot)$ on 
$I_n$.
Now, since $\sigma_{k_j}$ is 
bounded, we can choose a subsequence $\{k_{j'}\}$ of $\{k_j\}$ 
such that $\sigma_{k_{j'}} \to \sigma$ 
for some $\sigma \in [-1,1]$ as $j' \to \infty$. Thus, as $j' \to \infty$ 
$$t_0=nK_{pq}(k_{j'})+\asn_{pq}(\sigma_{k_{j'}})
\to \frac{n\pi_{\frac{p}{2},q}}{2}+
\arcsin_{\frac{p}{2},q}(\sigma),$$
and hence $\sigma=\sin_{\frac{p}{2},q}^{(n)}(t_0)$, 
which contradicts \eqref{eq:conti}. 
The proof to case $0<p \le 2$ is similar and we omit it.  
\end{proof}

\begin{rem}
In case $p>2$, 
$2K_{pq}(k)$ and $\sn_{pq}(\cdot,k)$ converge to the finite value 
$\pi_{\frac{p}{2},q}$ and to the finite-periodic function 
$\sin_{\frac{p}{2},q}$ as $k \to 1-0$, respectively. 
This is quite different from case $p=2$,
where $2K_{2q}(k)$ diverges to $\infty$ and 
$\sn_{22}{(t,k)}$ converges to the monotone increasing function
$\sin_{12}{(t)}=\tanh{(t)}$ as $k \to 1-0$.
\end{rem}

We define for $t \in [0,K_{pq}]$ (in case $0<p \le 1$, for $t \in [0,\infty)$)
\begin{align*}
\cn_{pq}{(t)}& :=(1-\sn_{pq}^q{(t)})^{\frac{1}{p}},\\
\dn_{pq}{(t)}& :=(1-k^q\sn_{pq}^q{(t)})^{\frac{1}{p}},
\end{align*}
then we obtain
\begin{align*}
&\cn_{pq}^p{(t)}+\sn_{pq}^q{(t)}=1,\\
&\frac{d}{dt}\sn_{pq}{(t)}=\cn_{pq}{(t)}\dn_{pq}{(t)}.
\end{align*}

\begin{prop}
For $p,\ q>1$, $\sn_{pq}$ satisfies 
for all $\mathbb{R}$
\begin{equation}
\label{eq:sn}
(\phi_p(u'))'+\frac{q}{p^*}\phi_q(u)(1+k^q-2k^q|u|^q)=0,
\end{equation}
which includes Eq.\,\eqref{eq:sin} as case $k=0$.
\end{prop}

\begin{proof}
For $t \in (0,K_{pq}(k))$ we have
\begin{align*}
(\phi_p(u'))'&=(\phi_p(\cn_{pq}{(t)}\dn_{pq}{(t)}))'\\
&=(((1-\sin^q_{pq}{(t)})(1-k^q\sin^q_{pq}{(t)}))^{\frac{1}{p^*}})'\\
&=\frac{1}{p^*}((1-\sin^q_{pq}{(t)})(1-k^q\sin^q_{pq}{(t)}))^{-\frac{1}{p}}\\
& \hspace{1cm} \times (-q \sin^{q-1}_{pq}{(t)}\cdot (1+k^q-2k^q\sin^{q}_{pq}{(t)}))
\cdot \cn_{pq}{(t)}\dn_{pq}{(t)}\\
&=-\frac{q}{p^*}\phi_q(u)(1+k^q-2k^q u^q).
\end{align*}
By symmetry of $\sn_{pq}$, Eq.\,\eqref{eq:sn} holds true 
for $t \neq t_n:=n K_{pq}(k),\ n \in \mathbb{Z}$.
Since $\lim_{t \to t_n} (\phi_p(u'))'$ exists, 
$\phi_p(u')$ is differentiable also at $t=t_n$ 
and satisfies Eq.\,\eqref{eq:sn} for all $\mathbb{R}$ in the classical sense.
\end{proof}

\begin{rem}
Letting $s=\sin_{pq}{(t)}$ in Eq.\,\eqref{eq:asn}, we have
\begin{align*}
\asn_{pq}{(\sigma,k)}=\int_0^{\as_{pq}{(\sigma)}}
\frac{dt}{\sqrt[p]{1-k^q\sin_{pq}^q{(t)}}}.
\end{align*}
We define the amplitude function $\am_{pq}(\cdot,k):[0,K_{pq}(k)] \to [0,\pi_{pq}/2]$ by 
\begin{align*}
t=\int_0^{\am_{pq}{(t,k)}} \frac{d\theta}{\sqrt[p]{1-k^q\sin_{pq}^q{(\theta)}}},
\end{align*}
thus $\sn_{pq}$ is represented by $\sin_{pq}$ as
\begin{align*}
\sn_{pq}{(t,k)}=\sin_{pq}{(\am_{pq}{(t,k)})}.
\end{align*}
\end{rem}

\section{Applications}

\subsection{The $(p,q)$-eigenvalue problem}

Let $T,\ \lambda>0$ and $p,\ q>1$. We consider 
the nonlinear eigenvalue problem
\begin{equation}
\label{eq:ep}
\begin{cases}
(\phi_p(u'))'+\lambda \phi_q(u)=0, & t \in (0,T),\\
u(0)=u(T)=0.
\end{cases} \tag{$\mathrm{E}_{pq}$}
\end{equation}
Problem \eqref{eq:ep} has been studied by many authors.
In particular, in paper \cite{O} of \^Otani, 
the existence of infinitely many multi-node solutions was
proved by using subdifferential operators method and phase-plane
analysis combined with symmetry properties of the solutions.
After that, Dr\'{a}bek and Man\'{a}sevich \cite{DM}
provided explicit forms of the whole spectrum and the corresponding
eigenfunctions for \eqref{eq:ep} (see also \cite{DKT}).
We follow \cite{DM} to understand completely the set of 
all solutions of \eqref{eq:ep}.

It will be convenient to find first the solution
to the initial value problem
\begin{equation}
\label{eq:iep}
\begin{cases}
(\phi_p(u'))'+\lambda \phi_q(u)=0,\\
u(0)=0,\ u'(0)=\alpha,
\end{cases}
\end{equation}
where without loss of generality we may assume $\alpha>0$.

Let $u$ be a solution to Eq.\,\eqref{eq:iep} and let $t(\alpha)$ be
the first zero point of $u'(t)$. On interval $(0,t(\alpha))$, $u$ satisfies
$u(t)>0$ and $u'(t)>0$, and thus
\begin{align*}
\frac{u'(t)^p}{p^*}+\lambda \frac{u(t)^q}{q}
=\lambda \frac{R^q}{q}=\frac{\alpha^p}{p^*},
\end{align*}
where $R=u(t(\alpha))>0$. Solving for $u'$ and integrating, we find
\begin{align*}
\left(\frac{q}{\lambda p^*}\right)^{\frac{1}{p}}
\int_0^t \frac{u'(s)}{(R^q-u(s)^q)^{\frac{1}{p}}}\,ds=t,
\end{align*}
which after a change of variable can be written as
\begin{align*}
t=\left(\frac{q}{\lambda p^*}\right)^{\frac{1}{p}}
\frac{1}{R^{\frac{q}{p}-1}} \int_0^{\frac{u(t)}{R}}
\frac{ds}{(1-s^q)^{\frac{1}{p}}}
=\left(\frac{q}{\lambda p^*}\right)^{\frac{1}{p}}
\frac{1}{R^{\frac{q}{p}-1}}
\as_{pq}{\left(\frac{u(t)}{R}\right)}.
\end{align*}
Thus we obtain the solution to Eq.\,\eqref{eq:iep}
can be written as
\begin{align}
\label{eq:gs}
u(t)=R\sin_{pq}{\left( 
\left(\frac{\lambda p^*}{q}\right)^{\frac{1}{p}}
R^{\frac{q}{p}-1} t
\right)},
\end{align}
where 
$R=(\frac{q}{\lambda p^*})^{\frac{1}{q}}
\alpha^{\frac{p}{q}}$.

\begin{thm}
All nontrivial solutions of \eqref{eq:ep} are given as follows.
For any given $R>0$, the set of eigenvalues of 
\eqref{eq:ep} is given by
\begin{align}
\label{eq:evR}
\lambda_n(R)
=\frac{q}{p^*} \left(\frac{n \pi_{pq}}{T}\right)^p R^{p-q}
\end{align}
for each $n \in \mathbb{N}$, with corresponding eigenfunctions
$\pm u_{n,R}$, where
\begin{align}
\label{eq:efR}
u_{n,R}(t)=R \sin_{pq}
{\left(\frac{n \pi_{pq}}{T}t\right)}.
\end{align}
\end{thm}

\begin{proof}
For given $R>0$, by imposing that $u$ in 
Eq.\,\eqref{eq:gs} satisfies the boundary conditions in \eqref{eq:ep},
we obtain that $\lambda$ is an eigenvalue of \eqref{eq:ep}
if and only if 
\begin{align*}
\left(\frac{\lambda p^*}{q}\right)^{\frac{1}{p}}
R^{\frac{q}{p}-1}T=n \pi_{pq}, \quad n \in \mathbb{N},
\end{align*}
and hence Eq.\,\eqref{eq:evR} follows. Expression \eqref{eq:efR}
for the 
eigenfunctions follows directly from Eq.\,\eqref{eq:gs}.
\end{proof}

\subsection{A perturbed $(p,q)$-eigenvalue problem}

Let $T,\ \lambda>0$ and $p,\ q>1$. We consider the 
nonlinear eigenvalue problem
\begin{equation}
\begin{cases}
(\phi_p(u'))'+\lambda \phi_q(u)(1-|u|^q)=0, & t \in (0,T),\\
u(0)=u(T)=0.
\end{cases} \tag{$\mathrm{PE}_{pq}$}
\end{equation}
Problem \eqref{eq:lep} has been studied by 
Berger and Fraenkel \cite{BF} and 
Chafee and Infante \cite{CI} ($p=q=2$), 
Wang and Kazarinoff \cite{WK} 
and Korman, Li and Ouyang \cite{KLO} ($p=2<q$),  
Guedda and V\'{e}ron \cite{GV} ($p=q>1$),
and Takeuchi and Yamada \cite{TY} ($p>2,\ q>1$).
However, there is no study providing explicit forms of the whole
spectrum and the corresponding eigenfunctions for \eqref{eq:lep}.

As we have done for \eqref{eq:ep}, 
it will be convenient to find first the solution
to the initial value problem
\begin{equation}
\label{eq:liep}
\begin{cases}
(\phi_p(u'))'+\lambda \phi_q(u)(1-|u|^q)=0,\\
u(0)=0,\ u'(0)=\alpha,
\end{cases}
\end{equation}
where without loss of generality we may assume $\alpha>0$.

Let $u$ be a solution to Eq.\,\eqref{eq:liep} and let $t(\alpha)$ be
the first zero point of $u'(t)$. On interval $(0,t(\alpha))$,
 $u$ satisfies
$u(t)>0$ and $u'(t)>0$, and thus
\begin{align*}
\frac{u'(t)^p}{p^*}+\lambda \frac{F(u)}{q}
=\lambda \frac{F(R)}{q}=\frac{\alpha^p}{p^*},
\end{align*}
where $F(s)=s^q-\frac{1}{2}s^{2q}$ and $R=u(t(\alpha))$. 
Since we are interested in functions satisfying 
the boundary condition of 
\eqref{eq:lep}, it suffices to assume $0<R \le 1$,
which means $|u| \le 1$. Moreover, we restrict
to $0<R<1$ and concentrate solutions satisfying
$|u|<1$ for a while. 

Solving for $u'$ and integrating, we find
\begin{align*}
\left(\frac{q}{\lambda p^*}\right)^{\frac{1}{p}}
\int_0^t \frac{u'(s)}{\sqrt[p]{F(R)-F(u(s))}}\,ds=t,
\end{align*}
which after a change of variable can be written as
\begin{align}
\label{eq:contradiction}
t&=\left(\frac{q}{\lambda p^*}\right)^{\frac{1}{p}}
\int_0^{\frac{u(t)}{R}}
\frac{R}{\sqrt[p]{F(R)-F(Rs)}}\,ds.
\end{align}
It is easy to verify that
$$F(R)-F(Rs)=F(R)(1-s^q)\left(1-\frac{R^q}{2-R^q}s^{q}\right),$$
and hence 
\begin{align*}
t&=\left(\frac{q}{\lambda p^*}\right)^{\frac{1}{p}}
\frac{R}{F(R)^{\frac{1}{p}}} 
\int_0^{\frac{u(t)}{R}}
\frac{ds}{\sqrt[p]{(1-s^q)(1-k^qs^q)}}
\quad \left(k^q:=\frac{R^q}{2-R^q}\right)\\
&=\left(\frac{q}{\lambda p^*}\right)^{\frac{1}{p}}
\frac{R}{F(R)^{\frac{1}{p}}}
\asn_{pq}{\left(\frac{u(t)}{R},k\right)}.
\end{align*}
Then we obtain that the solution to Eq.\,\eqref{eq:liep}
can be written as
\begin{align}
\label{eq:lgs}
u(t)=R\sn_{pq}{\left( 
\left(\frac{\lambda p^*}{q}\right)^{\frac{1}{p}}
\frac{F(R)^{\frac{1}{p}}}{R} t,k
\right)},
\end{align}
where
\begin{align}
\label{eq:k}
k&=\left(\frac{R^q}{2-R^q}\right)^{\frac{1}{q}},\\
\notag
R&=\left(
\frac{q}{\lambda p^*}\right)^{\frac{1}{q}}
\alpha^{\frac{p}{q}}
\left(\frac{1}{2}+\frac{1}{2}\sqrt{1-
\frac{2q}{\lambda p^*}\alpha^p}\right)^{-\frac{1}{q}}.
\end{align}

We first observe the structure of the set of all 
nontrivial solutions of \eqref{eq:lep} satisfying $|u|<1$.

\begin{thm}[$|u|<1$]
\label{thm:R<1}
All nontrivial solutions for $p \in (1,2]$ and 
all nontrivial solutions with $|u|<1$ for $p>2$ are given as follows.
For any given $k \in (0,1)$, 
the set of eigenvalues of 
\eqref{eq:lep} is given by
\begin{align}
\label{eq:lev}
\lambda_n(k)=
\frac{q}{p^*} (1+k^q) 
\left(\frac{2k^q}{1+k^q}\right)^{\frac{p}{q}-1}
\left(\frac{2nK_{pq}(k)}{T}\right)^p
\end{align}
for each $n \in \mathbb{N}$, with corresponding eigenfunctions
$\pm u_{n,k}$, where
\begin{align}
\label{eq:lef}
u_{n,k}(t)=
\left(\frac{2k^q}{1+k^q}\right)^{\frac{1}{q}}
\sn_{pq}{\left(\frac{2nK_{pq}(k)}{T}t,k\right)}.
\end{align}
\end{thm}

\begin{proof}
For $k \in (0,1)$ given, we impose that Function
\eqref{eq:lgs} with $R \in (0,1)$ decided from Eq.\,\eqref{eq:k} 
satisfies the boundary conditions in \eqref{eq:lep}.
Then, we obtain that $\lambda$ is an eigenvalue of \eqref{eq:lep}
if and only if
\begin{align*}
\left(\frac{\lambda p^*}{q}\right)^{\frac{1}{p}}
\frac{F(R)^{\frac{1}{p}}}{R}T
=2nK_{pq}(k), \quad n \in \mathbb{N}.
\end{align*}
From Eq.\,\eqref{eq:k} again we have
\begin{align*}
\frac{F(R)^{\frac{1}{p}}}{R}=
\left(\frac{2k^q}{1+k^q}\right)^{\frac{1}{p}-\frac{1}{q}}
(1+k^q)^{-\frac{1}{p}},
\end{align*}
and hence we obtain Eq.\,\eqref{eq:lev}. Expression 
\eqref{eq:lef} for the eigenfunctions follows
then directly from Eq.\,\eqref{eq:lgs}.

It remains to show that no other nontrivial solution of
\eqref{eq:lep} is 
obtained when $1<p \le 2$. Assume the contrary.  
Then there exist $t_*>0$ and  
a nontrivial solution $u$ of \eqref{eq:lep} with $R=u(t_*)=1$.
However, the right-hand side of Eq.\,\eqref{eq:contradiction}
with $t=t_*$ diverges because
 $\sqrt[p]{F(1)-F(s)}=O((1-s^q)^{\frac{2}{p}})$
as $s \to 1-0$. Thus, $t_*=\infty$, which is a contradiction. 
\end{proof}

Next we find solutions of \eqref{eq:lep} with $|u| \le 1$,
except the solutions given by Theorem \ref{thm:R<1}.
From Proposition \ref{prop:K}, one of solutions of 
Eq.\,\eqref{eq:liep} is obtained
by $k \to 1-0$ in Eq.\,\eqref{eq:lgs} with 
Eq.\,\eqref{eq:k}, namely 
\begin{align*}
u(t)=
\sin_{\frac{p}{2},q}
{\left(\left(\frac{\lambda p^*}{2q}\right)^{\frac{1}{p}}t\right)}. 
\end{align*}
Now we assume $p>2$ and take a number $t_*$ as 
$(\frac{\lambda p^*}{2q})^{\frac{1}{p}}t_*=\pi_{\frac{p}{2},q}/2$,
then $u$ attains $1$ at $t=t_*$ (note that 
it is impossible to obtain such a solution when $1<p \le 2$).   
Using this $u$, 
we can make the other solutions of Eq.\,\eqref{eq:liep} as follows.
In the phase-plane,
the orbit $(u(t),u'(t))$ arrives at the 
equilibrium point $(1,0)$ at $t=t_*$ and 
can stay there for any finite time $\tau$ 
before it begins to leave there.
Then, the interval $[t_*,t_*+\tau]$ is a flat core of the solution.
Similarly, there is the other
equilibrium point $(-1,0)$, where the orbit can stay,
and the solution has another flat core of any finite length.
Thus we have solutions of Eq.\,\eqref{eq:liep} 
attaining $\pm 1$ with any number of flat cores.
\begin{thm}[$|u| \le 1$]
\label{thm:R=1}
Let $p>2$, then all nontrivial solutions are given as follows, 
in addition to Theorem \ref{thm:R<1}.
For any given $\tau \in [0,T)$, 
the set of eigenvalues of \eqref{eq:lep}
is given by
\begin{align*}
\Lambda_n(\tau)=\frac{2q}{p^*}
\left(\frac{n\pi_{\frac{p}{2},q}}{T-\tau}\right)^p
\end{align*}
for each $n \in \mathbb{N}$, with corresponding eigenfunctions
$\pm u_{n,\{\tau_i\}}$, where $u_{n,\{\tau_i\}}$ is  
any function given as follows: 
for any 
$\{\tau_i\}_{i=1}^{n}$ with 
$\tau_i \ge 0$ and $\sum_{i=1}^n \tau_i=\tau$   
\begin{align}
\label{eq:flatcore}
u_{n,\{\tau_i\}}(t)=
\begin{cases}
(-1)^{j-1} \sin_{\frac{p}{2},q}
{\left(\frac{n\pi_{\frac{p}{2},q}}{T-\tau}(t-T_{j-1})\right)}
& \text{if}\;\ T_{j-1} \le t \le T_{j-1}+\frac{T-\tau}{2n},\\
(-1)^{j-1} 
& \text{if}\;\ T_{j-1}+\frac{T-\tau}{2n}
\le t \le T_{j}-\frac{T-\tau}{2n},\\
(-1)^{j-1} \sin_{\frac{p}{2},q}{\left(\frac{n\pi_{\frac{p}{2},q}}{T-\tau}
(T_j-t)\right)}
& \text{if}\;\ T_{j}-\frac{T-\tau}{2n}
\le t \le T_j,\\
j=1,2,\ldots,n,
\end{cases}
\end{align}
where $T_0=0$ and 
$T_j=\frac{(T-\tau)j}{n}
+\sum_{i=1}^j \tau_i$ for $j=1,2,\ldots,n$. 
\end{thm}

\begin{proof}
For each $n \in \mathbb{N}$, 
it suffices to construct solutions with $(n-1)$-zeros.
Let $\tau \in [0,T)$.
They are all generated by the eigenvalue and the corresponding
eigenfunction of \eqref{eq:lep} with $T$ replaced by $T-\tau$
\begin{align*}
&\Lambda_n(\tau)=\frac{2q}{p^*}
\left(\frac{n\pi_{\frac{p}{2},q}}{T-\tau}\right)^p,\\
&u_{n,\tau}(t)=
\sin_{\frac{p}{2},q}\left(\frac{n\pi_{\frac{p}{2},q}}{T-\tau}t\right),
\end{align*}
which are obtained from Eqs.\,\eqref{eq:lev} and \eqref{eq:lef} 
with $k \to 1-0$, respectively. In the phase-plane,
the orbit $(u_{n,\tau}(t),u_{n,\tau}'(t))$ goes through 
the equilibrium points $(\pm 1,0)$ in $n$-times 
without staying there as $t$ increases from $0$ to $T-\tau$. 
Therefore, if the orbit stays the $i$-th equilibrium point for time 
$\tau_i$, where $\tau_1+\tau_2+\cdots+\tau_n=\tau$, 
then we can obtain Solution \eqref{eq:flatcore}
with $n$-flat cores in $[0,T]$. 
\end{proof}

In Theorems \ref{thm:R<1} and \ref{thm:R=1}, we give parameters 
$k$ and $\tau$ to obtain the eigenvalue and the corresponding eigenfunction
of \eqref{eq:lep}. Conversely,
giving any $\lambda>0$, we can observe the set $S_{\lambda}$
of all solutions of \eqref{eq:lep}
by considering the inverses of $\lambda_n$
and $\Lambda_n$. 

\begin{thm}
Let $p>1$ and $q>1$.

Case $p>q$. For any $\lambda>0$ there exists a strictly decreasing positive sequence 
$\{k_j\}_{j=1}^{\infty}$ such that $k_j \to 0$ as $j \to \infty$ and 
$$S_\lambda=\{0\} \cup \bigcup_{j=1}^{\infty}\{\pm u_{j,k_j}\}.$$

Case $p=q$. If 
\begin{align*}
0<\lambda \le
\frac{q}{p^*}\left(\frac{\pi_{pq}}{T}\right)^p,
\end{align*}
then $S_\lambda=\{0\}$. If 
\begin{align*}
\frac{q}{p^*}\left(\frac{n \pi_{pq}}{T}\right)^p
<\lambda \le \frac{q}{p^*}\left(\frac{(n+1) \pi_{pq}}{T}\right)^p,
\quad n \in \mathbb{N},
\end{align*}
then there exists a strictly decreasing positive sequence 
$\{k_j\}_{j=1}^{n}$ such that 
$$S_\lambda=\{0\} \cup \bigcup_{j=1}^{n}\{\pm u_{j,k_j}\}.$$

Case $p<q$. There exists $\lambda_1>0$ such that if 
$0<\lambda <\lambda_1$, then $S_\lambda=\{0\}$.
If $n^p\lambda_1 \le \lambda <(n+1)^p\lambda_1,\ n \in \mathbb{N}$,
then there exist a strictly decreasing positive sequence 
$\{k_j\}_{j=1}^{n}$ and a strictly increasing positive sequence
$\{\ell_j\}_{j=1}^{n}$ such that 
$k_j>\ell_j,\ j=1,2,\ldots,n-1$ and
$$S_\lambda=\{0\} \cup \bigcup_{j=1}^{n}\{\pm u_{j,k_j}\}
\cup \bigcup_{j=1}^{n} \{\pm u_{j,\ell_j}\},$$
where $u_{n,k_n}=u_{n,\ell_n}$ with $k_n=\ell_n$
for $\lambda=n^p \lambda_1$ and 
$|u_{n,k_n}|> |u_{n,\ell_n}|\ 
(t \neq jT/n,\ j=1,2,\ldots,n-1)$ 
with $k_n>\ell_n$ otherwise.

In any case, each $k_j,\ \ell_j$ is calculated by Eq.\,\eqref{eq:lev} 
for $\lambda_j$, and the corresponding solution is given 
in Form \eqref{eq:lef}.

When $1<p \le 2$, we have $k_j<1$. 
When $p>2$, in addition, if
$$\lambda \ge \frac{2q}{p^*}\left(\frac{m\pi_{\frac{p}{2},q}}{T}\right)^p,
\quad m \in \mathbb{N},$$
then for each $j=1,2,\ldots,m$, 
the set $\{\pm u_{j,k_j}\}$ above is replaced by 
$\cup_{\{\tau_i\}} \{\pm u_{j,\{\tau_i\}}\}$,
where $\cup_{\{\tau_i\}}$ is the union for all $\{\tau_i\}_{i=1}^j$ 
satisfying $\tau_i \ge 0$ and 
$$\sum_{i=1}^j \tau_i=T-j\pi_{\frac{p}{2},q} 
\left(\frac{2q}{\lambda p^*}\right)^{\frac{1}{p}}.$$
The nontrivial solution $u_{j,\{\tau_i\}}$ is
given in Form \eqref{eq:flatcore}.
\end{thm}

\begin{proof}
First we assume $1<p \le 2$.
In this case, we have already known that 
all nontrivial solutions of \eqref{eq:lep} are
obtained by Theorem \ref{thm:R<1}.

Now we fix $\lambda>0$. 
We obtain that $\lambda$ is the $j$-th eigenvalue
of \eqref{eq:lep} if and only if from Eq.\,\eqref{eq:lev}
there exists $k \in (0,1)$ such that
\begin{align}
\label{eq:key}
\frac{T}{2j}\left(\frac{\lambda p^*}{q}\right)^{\frac{1}{p}}
=(1+k^q)^{\frac{1}{p}}\left(\frac{2k^q}{1+k^q}\right)
^{\frac{1}{q}-\frac{1}{p}}K_{pq}(k)=:\Phi(k).
\end{align}

Case $p>q$.
$\Phi(k)$ is 
strictly increasing in $(0,1)$ and it follows 
from Proposition \ref{prop:K} that  
$\Phi(0)=0$ and $\lim_{k \to 1-0}\Phi(k)=\infty$.
Thus, there exists a unique $k=k_j(\lambda)$ satisfying   
Eq.\,\eqref{eq:key}.
For $j$ and $k_j$, a unique 
solution $u_{j,k_j}$ of \eqref{eq:lep}
is obtained by Eq.\,\eqref{eq:lef}.

Case $p=q$.
$\Phi(k)$ is 
strictly increasing in $(0,1)$ and it 
follows from Proposition \ref{prop:K} that
$\Phi(0)=\pi_{pq}/2$ and $\lim_{k \to 1-0}\Phi(k)
=\infty$.
Thus, if
\begin{align*}
\frac{T}{2j}\left(\frac{\lambda p^*}{q}\right)^{\frac{1}{p}}
>\frac{\pi_{pq}}{2},
\end{align*}
namely,
\begin{align*}
\lambda>\frac{q}{p^*}\left(\frac{j \pi_{pq}}{T}\right)^p,
\end{align*}
then there exists a unique $k=k_j(\lambda)$ satisfying   
Eq.\,\eqref{eq:key}.
For $j$ and $k_j$, a unique 
solution $u_{j,k_j}$ of \eqref{eq:lep}
is obtained by Eq.\,\eqref{eq:lef}.
 
Case $p<q$. 
It is clear that  
$\lim_{k \to +0}\Phi(k)=\lim_{k \to 1-0}\Phi(k)=\infty$.
Changing variable $r=\frac{k^q}{1+k^q}$, we can 
write $\Phi$ as
\begin{align*}
\Psi(r)=
\int_0^1 \frac{(1+s^q)^{\frac{1}{p}-\frac{1}{q}}}
{(1-s^q)^{\frac{1}{p}}}
\psi((1+s^q)r)\,ds, \quad r \in (0,1/2),
\end{align*}
where $\psi(t)=t^{\frac{1}{q}-\frac{1}{p}}(1-t)^{-\frac{1}{p}}$.
It is easy to see that $\psi$ is convex in $(0,1)$ 
because $\psi(t)>0$ and
\begin{align*}
(\log{\psi(t)})''=\left(\frac{1}{p}-\frac{1}{q}\right)\frac{1}{t^2}
+\frac{1}{p}\frac{1}{(1-t)^2}>0.
\end{align*}
Then, $\Psi$ is twice-differentiable in $(0,1/2)$ and
\begin{align*}
\Psi''(r)=
\int_0^1 \frac{(1+s^q)^{\frac{1}{p}-\frac{1}{q}+2}}
{(1-s^q)^{\frac{1}{p}}}
\psi''((1+s^q)r)\,ds>0.
\end{align*}
Thus, $\Psi$ is convex and there exists $k_* \in (0,1)$ such that 
$\Phi(k_*)$ is the only one critical value,
and hence the minimum of $\Phi$ in $(0,1)$. 

If 
\begin{align*}
\frac{T}{2j}\left(\frac{\lambda p^*}{q}\right)^{\frac{1}{p}}=\Phi(k_*),
\end{align*}
namely,
\begin{align*}
\lambda=j^p\lambda_1:=
\frac{q}{p^*}\left(\frac{2j \Phi(k_*)}{T}\right)^p,
\end{align*}
then $k_*$ satisfies Eq.\,\eqref{eq:key}.
For $j$ and $k_*$, a unique 
solution $u_{j,k_*}$ of \eqref{eq:lep}
is obtained by Eq.\,\eqref{eq:lef}.
Moreover, if 
\begin{align*}
\frac{T}{2j}\left(\frac{\lambda p^*}{q}\right)^{\frac{1}{p}}
>\Phi(k_*),
\end{align*}
namely, $\lambda>j^p\lambda_1$,
then there exist $k=k_j(\lambda)$ and $\ell_j(\lambda)$ such that   
\begin{align*}
k_j(\lambda)&=\Phi^{-1}\left(
\frac{T}{2j}\left(\frac{\lambda p^*}{q}\right)^{\frac{1}{p}}
\right) \in (k_*,1),\\
\ell_j(\lambda)&=\Phi^{-1}\left(
\frac{T}{2j}\left(\frac{\lambda p^*}{q}\right)^{\frac{1}{p}}
\right) \in (0,k_*).
\end{align*}
For $j$, $k_j$ and $\ell_j$, 
solutions $u_{j,k_j}$ and $u_{j,\ell_j}$ of \eqref{eq:lep}
are obtained by Eq.\,\eqref{eq:lef}.

Next, we assume $p>2$. In any case, a similar proof as above with 
$\lim_{k \to 1-0} \Phi(k)=2^{\frac{1}{p}-1}\pi_{\frac{p}{2},q}$
instead of $\lim_{k \to 1-0}\Phi(k)=\infty$
gives that it is impossible to find $k_m \in (0,1)$ above
satisfying Eq.\,\eqref{eq:key}, provided
$$\lambda \ge  \frac{2q}{p^*}\left(\frac{m \pi_{\frac{p}{2},q}}
{T}\right)^p, \quad m \in \mathbb{N}.$$
Then, however, for each $j=1,2,\ldots,m$, we can take 
$\tau \in [0,T)$ such that 
\begin{align*}
\lambda=\frac{2q}{p^*}
\left(\frac{j\pi_{\frac{p}{2},q}}{T-\tau}\right)^p,
\end{align*}
and Theorem \ref{thm:R=1} yields the solutions 
$u_{j,\{\tau_i\}}$,
where $\{\tau_i\}_{i=1}^j$ is any sequence satisfying that
$\tau_i \ge 0,\ \sum_{i=1}^j \tau_i=\tau$.
\end{proof}

It follows directly from Representation
\eqref{eq:flatcore} of Theorem \ref{thm:R=1}
that a kind of solution of \eqref{eq:lep} with $p$-Laplacian
is also an eigenfunction of $(\mathrm{E}_{\frac{p}{2},q})$
with $p/2$-Laplacian.

\begin{cor}
Let $p>2$. For each $n \in \mathbb{N}$, 
any solution $u_{n,\{\tau_i\}}$ 
of \eqref{eq:lep} in Theorem \ref{thm:R=1} satisfies
\begin{align*}
(\phi_{\frac{p}{2}}(u'))'
+\frac{(p-2)q}{p}
\left(\frac{n\pi_{\frac{p}{2},q}}{T-\tau}\right)^{\frac{p}{2}}
\phi_q(u)=0
\end{align*}
in the intervals where $|u_{n,\{\tau_i\}}|<1$, 
where $\tau=\sum_{i=1}^n \tau_i$.
In particular, for each $n \in \mathbb{N}$, 
the solution $u_{n,\{0\}}$ of \eqref{eq:lep} with $\tau=0$
is an eigenfunction of $(\mathrm{E}_{\frac{p}{2},q})$, 
that is,
\begin{align*}
\begin{cases}
(\phi_{\frac{p}{2}}(u'))'+ \frac{(p-2)q}{p}
\left(\frac{n\pi_{\frac{p}{2},q}}{T}\right)^{\frac{p}{2}}
\phi_q(u)=0, & t \in (0,T),\\
u(0)=u(T)=0.
\end{cases}
\end{align*}
Moreover, $u_{n,\{0\}}$ is characterized by $u_{n,R}$ with 
$R=1$ in Eq.\,\eqref{eq:efR} with $p$ replaced by $p/2$.
\end{cor}









\end{document}